\IfFileExists{proc-l.cls}%
  {\documentclass{proc-l}}%
  {\documentclass{amsart}}

\usepackage{amssymb}
\usepackage[all,2cell]{xy}
\usepackage{hyperref}

\newtheorem{teo}{Theorem}[section]
\newtheorem{cor}[teo]{Corollary}
\newtheorem{lem}[teo]{Lemma}
\newtheorem{defi}[teo]{Definition}
\newtheorem{prop}[teo]{Proposition}
\newtheorem{example}[teo]{Example}
\newtheorem{remark}[teo]{Remark}

\begin{document}

\title{Torsion and exactness}

\author{Marino Gran}
\address[Marino Gran]{Institut de Recherche en Math\'ematique et Physique, Universit\'e catholique de Louvain, Chemin du Cyclotron 2, 1348 Louvain-la-Neuve, Belgique}
\email{marino.gran@uclouvain.be}
\thanks{The first author's research was supported by the Fonds de la Recherche Scientifique -- FNRS under Grant CDR No.~J.0092.26}

\author{George Janelidze}
\address[George Janelidze]{Department of Mathematics and Applied Mathematics, University of Cape Town, Rondebosch 7700, South Africa}
\email{george.janelidze@uct.ac.za}
\thanks{The second author gratefully acknowledges the financial support of UCLouvain, which made possible a research visit in July 2026 during which part of this work was carried out}


\subjclass[2020]{Primary 18E40; Secondary 18A30, 18A40}

\keywords{Non-pointed torsion theory, pretorsion theory, near-torsion theory, kernel diagram, cokernel diagram, exact sequence, reflective subcategory, coreflective subcategory}

\date{\today}

\begin{abstract}
We give an equivalent definition of a non-pointed torsion theory, also called a pretorsion theory in the literature. This uses a new, more general, notion of a near-torsion theory, and does not use non-pointed exactness. Furthermore, a general method of associating the `largest' (non-pointed) torsion theory to any given near-torsion theory is described. We also describe a near-torsion theory formed by the class of cokernel diagrams and the class of kernel diagrams in a pointed category with kernels and cokernels, and its associated torsion theory.
\end{abstract}

\maketitle

\section{Introduction}

Torsion theories were first defined and investigated in abelian
categories by Dickson in \cite{[Dic1966]}.
A torsion theory is a pair $(\mathcal T, \mathcal F)$ of full replete subcategories of an abelian category $\mathcal C$ satisfying the following properties:
\begin{itemize}
\item for any object $A \in \mathcal C$, there is a short exact sequence
\begin{equation}\label{Short} \xymatrix{0 \ar[r] & t(A)\ar[r]&A\ar[r]&f(A) \ar[r] & 0} 
\end{equation}
with $t(A) \in \mathcal T$ and $f(A) \in \mathcal F$;
\item $\mathsf{Hom}_{\mathcal C}(T,F) = \{ 0 \}$, for any $T \in \mathcal T$ and $F \in \mathcal F$.
\end{itemize}

Over the last twenty years, many
authors have studied various properties of torsion theories in
more general pointed contexts, and this has led to the discovery of many new results and examples, in the categories of topological groups \cite{[BG2006]}, commutative
rings \cite{[CDT2006]}, crossed modules \cite{[EG2013]}, and cocommutative
Hopf algebras \cite{[GKV2016]}, among others.
New characterizations of torsion and torsion-free subcategories, in terms of
their stability under extensions, were established in \cite{[JT2007]} for
general pointed categories with kernels and cokernels, pointing to some important differences with respect to the classical abelian context. More recently, even the
assumption of the existence of a zero object has been dropped, and various
approaches to \emph{non-pointed torsion theories} have been proposed (see
\cite{[FF2020], [GJ2020], [FFG2021], [X2022], [BCGT2023], [BCG2025], [CF2025],
[MM2026]} and the references therein). We refer the reader to the Introduction
of \cite{[CM2026]} for a nice overview of these recent developments of
torsion theories in pointed and non-pointed categories.

The main definition of a non-pointed torsion theory we will use in this paper is equivalent to the one of \cite{[FF2020]}, whose main properties were established in \cite{[FFG2021]}. The crucial difference with the classical notion of torsion theory in an abelian category recalled above is in replacing the zero object with the full replete subcategory $ \mathcal Z = \mathcal T \cap \mathcal F$ of \emph{trivial objects} of $\mathcal C$, and to define the ``exactness'' of the  sequence \eqref{Short} with respect to the ideal of \emph{trivial morphisms} \cite{[Ehr1964]}, which are the ones which factor through an object in $\mathcal Z$. The choice of defining kernels and cokernels with respect to a general ideal of morphisms is often referred to as 
 \emph{non-pointed exactness}, and this has played a role in a specific approach to non-abelian homological algebra \cite{[G2012]}.

Note, that the papers \cite{[FF2020],[FFG2021]} and several others, mentioned above, say ``pretorsion theory'' instead of ``non-pointed torsion theory''.

The purpose of this paper is three-fold:
\begin{itemize}
	\item [(i)] To give a definition of a non-pointed torsion theory that does not use non-pointed exactness, where, however, a new, more general, notion of a near-torsion theory is used. Every near-torsion theory $(\mathcal{T},\mathcal{F})$ in a category $\mathcal{C}$ has the `largest' triple $(\mathcal{C}_{\mathcal{T},\mathcal{F}},\mathring{\mathcal{T}},\mathring{\mathcal{F}})$ such that $\mathcal{C}_{\mathcal{T},\mathcal{F}}\subseteq\mathcal{C}$, $\mathring{\mathcal{T}}\subseteq\mathcal{T}$, $\mathring{\mathcal{F}}\subseteq\mathcal{F}$, and $(\mathring{\mathcal{T}},\mathring{\mathcal{F}})$ is a non-pointed torsion theory in $\mathcal{C}_{\mathcal{T},\mathcal{F}}$.
	\item [(ii)] To establish a non-pointed-torsion-theoretic approach to the `classical' (that is, pointed) concept of exactness. What we mean by this is to describe a near-torsion theory $(\mathcal{T},\mathcal{F})$ with $\mathcal{T}$ being the class of cokernel diagrams and $\mathcal{F}$ being the class of kernel diagrams in a given pointed category $\mathcal{A}$ with kernels and cokernels.
	\item [(iii)] Furthermore, to explain that the construction of the above-mentioned triple $(\mathcal{C}_{\mathcal{T},\mathcal{F}},\mathring{\mathcal{T}},\mathring{\mathcal{F}})$ suggests a refined definition of exactness of a sequence of morphisms in a pointed category with kernels and cokernels. 
\end{itemize}

We think this short paper will also initiate several directions of further study, as very briefly indicated in our last section.

\section{Near-torsion theories and non-pointed torsion theories}

\begin{defi}
	\emph{A pair $(\mathcal{T},\mathcal{F})$ of full replete subcategories of a category $\mathcal{C}$ will be called a} near-torsion theory \emph{in $\mathcal{C}$ if it satisfies the following conditions:}
	\begin{itemize}
		\item [(a)] \emph{$\mathcal{T}$ is a coreflective subcategory of $\mathcal{C}$, that is, the inclusion functor $\mathcal{T}\to\mathcal{C}$ has a left inverse right adjoint $t:\mathcal{C}\to\mathcal{T}$;}
		\item [(b)] \emph{$\mathcal{F}$ is a reflective subcategory of $\mathcal{C}$, that is, the inclusion functor $\mathcal{F}\to\mathcal{C}$ has a left inverse left adjoint $f:\mathcal{C}\to\mathcal{F}$;}
		\item [(c)] \emph{$A\in\mathcal{F}\Rightarrow t(A)\in\mathcal{F}$;}
		\item [(d)] \emph{$A\in\mathcal{T}\Rightarrow f(A)\in\mathcal{T}$.}
	\end{itemize}
\end{defi}
With the notation of Definition 2.1, we will sometimes write $$(\mathcal{T},\mathcal{F})=(\mathcal{C},\mathcal{T},\mathcal{F},t,\varepsilon,f,\eta),$$ where
$$(\text{Inclusion},t,\text{Identity natural transformation},\varepsilon):\mathcal{T}\to\mathcal{C},$$ $$(f,\text{Inclusion},\eta,\text{Identity natural transformation}):\mathcal{C}\to\mathcal{F}$$ are the adjunctions involved. Accordingly, we have $T\in\mathcal{T}\Leftrightarrow\varepsilon_T=1_T$ and $F\in\mathcal{F}\Leftrightarrow\eta_F=1_F.$
\begin{remark}
	\emph{As follows from 2.1(c) and 2.1(d), and idempotency of $t$ and $f$, the following conditions on an object $Z$ in $\mathcal{C}$ are equivalent:}
	\begin{itemize}
		\item [(a)] \emph{$Z\in\mathcal{T}\cap\mathcal{F}$;}
		\item [(b)] \emph{$Z=t(F)$ for some $F$ in $\mathcal{F}$;}
		\item [(c)] \emph{$Z=f(T)$ for some $T$ in $\mathcal{T}$;}
		\item [(d)] \emph{$Z=tf(A)$ for some $A$ in $\mathcal{C}$;}
		\item [(e)] \emph{$Z=ft(A)$ for some $A$ in $\mathcal{C}$.}
	\end{itemize} 
\end{remark}
\begin{lem}
	For a near-torsion theory $(\mathcal{C},\mathcal{T},\mathcal{F},t,\varepsilon,f,\eta)$ and an object $B$ in $\mathcal{C}$ such that $\varepsilon_{f(B)}:tf(B)\to f(B)$ is a monomorphism, the following conditions are equivalent:
	\begin{itemize}
		\item [(a)] $\varepsilon_B:t(B)\to B$ is a $(\mathcal{T}\cap\mathcal{F})$-kernel of $\eta_B:B\to f(B)$, that is, $\eta_B\varepsilon_B$ factors through an object in $\mathcal{T}\cap\mathcal{F}$, and, for every morphism $\alpha:A\to B$ such that $\eta_B\alpha$ factors through an object in $\mathcal{T}\cap\mathcal{F}$, there exists a unique morphism $\alpha':A\to t(B)$ with $\varepsilon_B\alpha'=\alpha$;
		\item [(b)] $\varepsilon_B$ is a monomorphism and a morphism $\alpha:A\to B$ factors through an object in $\mathcal{T}$ whenever the composite $\eta_B\alpha$ does;
		\item [(c)] the diagram $$\xymatrix{t(B)\ar[d]_{\varepsilon_B}\ar[r]^-{t(\eta_B)}&tf(B)\ar[d]^{\varepsilon_{f(B)}}\\B\ar[r]_-{\eta_B}&f(B)}$$ is a pullback. 
	\end{itemize}
\end{lem}
\begin{proof} 	
	(a)$\Rightarrow$(b): Suppose (a) holds, and consider the commutative diagram $$\xymatrix{&A\ar@{.>}[dl]\ar[d]^\alpha\ar[r]^-\gamma&T\ar[d]^\beta\\t(B)\ar[r]_-{\varepsilon_B}&B\ar[r]_-{\eta_B}&f(B)}$$  
	of solid arrows, in which $\eta_B\alpha=\beta\gamma$ represents a factorization of $\eta_B\alpha$ through an object $T$ of $\mathcal{T}$. Since $T$ belongs to $\mathcal{T}$ and $f(B)$ belongs to $\mathcal{F}$, the morphism $\beta:T\to f(B)$ factors through $tf(B)$, which belongs to $\mathcal{T}\cap\mathcal{F}$. Therefore $\eta_B\alpha=\beta\gamma$ also factors through an object of $\mathcal{T}\cap\mathcal{F}$. This gives the dotted arrow showing that $\alpha$ factors through an object in $\mathcal T$. The fact that $\varepsilon_B$ is a monomorphism follows from the uniqueness requirement for $\alpha'$ in $(a)$.
	
	(b)$\Rightarrow$(a): The morphism $\eta_B \varepsilon_B$ factors through $tf(B)$, which belongs to $\mathcal T \cap \mathcal F$ by definition of near-torsion theory. Suppose then that (b) holds, and consider the commutative diagram $$\xymatrix{T\ar@{.>}[d]\ar[dr]|{\beta'}&A\ar[l]_-{\gamma'}\ar[d]^\alpha\ar[r]^-\gamma&Z\ar[d]^\beta\\t(B)\ar[r]_-{\varepsilon_B}&B\ar[r]_-{\eta_B}&f(B)}$$  
	of solid arrows, in which $\eta_B\alpha=\beta\gamma$ represents a factorization of $\eta_B\alpha$ through an object $Z$ of $\mathcal{T}\cap\mathcal{F}\subseteq\mathcal{T}$, and $\alpha=\beta'\gamma'$ represents a factorization of $\alpha$ through an object in $\mathcal{T}$ (which exists by (b)). Since $T$ belongs to $\mathcal{T}$ this gives us the dotted arrow that keeps the diagram commutative. To prove (a) we take the desired $\alpha':A\to t(B)$ to be the composite of the dotted arrow with $\gamma'$. The uniqueness of $\alpha'$ follows from the fact that $\varepsilon_B$ is a monomorphism.
	
	(a)$\Leftrightarrow$(c): Just consider the diagram  $$\xymatrix{A\ar@/^-0.5cm/[ddr]_{\alpha}\ar@{.>}[dr]\ar@/^0.5cm/[drr]^{\alpha'}\\&t(B)\ar[d]_{\varepsilon_B}\ar[r]^-{t(\eta_B)}&tf(B)\ar[d]^{\varepsilon_{f(B)}}\\&B\ar[r]_-{\eta_B}&f(B)}$$  
	and note that, since $\varepsilon_{f(B)}$ is a monomorphism, its upper triangle commutes whenever the lower triangle and the square do.   
\end{proof}
Dually, we obtain:
\begin{lem}
	For a near-torsion theory $(\mathcal{C},\mathcal{T},\mathcal{F},t,\varepsilon,f,\eta)$ and an object $B$ in $\mathcal{C}$ such that $\eta_{t(B)}:t(B)\to ft(B)$ is an epimorphism, the following conditions are equivalent:
	\begin{itemize}
		\item [(a)] $\eta_B:B\to f(B)$ is a $(\mathcal{T}\cap\mathcal{F})$-cokernel of $\varepsilon_B:t(B)\to B$, that is, $\eta_B\varepsilon_B$ factors through an object in $\mathcal{T}\cap\mathcal{F}$, and, for every morphism $\beta:B\to C$ such that $\beta\varepsilon_B$ factors through an object in $\mathcal{T}\cap\mathcal{F}$, there exists a unique morphism $\beta':f(B)\to C$ with $\beta'\eta_B=\beta$;
		\item [(b)] $\eta_B$ is an epimorphism and a morphism $\beta:B\to C$ factors through an object in $\mathcal{F}$ whenever the composite $\beta\varepsilon_B$ does;
		\item [(c)] the diagram $$\xymatrix{t(B)\ar[d]_{\varepsilon_B}\ar[r]^-{\eta_{t(B)}}&ft(B)\ar[d]^{f(\varepsilon_B)}\\B\ar[r]_-{\eta_B}&f(B)}$$ is a pushout.\qed 
	\end{itemize}
\end{lem}
\begin{defi}
	\emph{A near-torsion theory $(\mathcal{T},\mathcal{F})=(\mathcal{C},\mathcal{T},\mathcal{F},t,\varepsilon,f,\eta)$ is said to be a} (non-pointed) torsion theory \emph{in $\mathcal{C}$ if, for each object $B$ in $\mathcal{C}$, all the conditions of Lemmas 2.3 and 2.4 hold, or, equivalently:}
	\begin{itemize}
		\item [(a)] \emph{$\varepsilon_{f(B)}:tf(B)\to f(B)$ is a monomorphism and the equivalent conditions of Lemma 2.3 hold;}
		\item [(b)] \emph{$\eta_{t(B)}:t(B)\to ft(B)$ is an epimorphism and the equivalent conditions of Lemma 2.4 hold.}
	\end{itemize}
\end{defi}
\begin{remark}
	\emph{As conditions 2.3(a) and 2.4(a) easily show, what we call a non-pointed torsion theory is the same as what is called a pretorsion theory, originally in \cite{[FF2020]} and then in \cite{[FFG2021]}, as well as in some more recent papers (see \cite{[BCG2025], [BCGT2023], [CF2025], [MM2026], [X2022], [FF2026]} for instance, and the references therein).} 
\end{remark}
\begin{teo}\label{largest}
	For a near-torsion theory $(\mathcal{T},\mathcal{F})=(\mathcal{C},\mathcal{T},\mathcal{F},t,\varepsilon,f,\eta)$, let $\mathring{\mathcal{T}}$, $\mathring{\mathcal{F}}$, and $\mathcal{C}_{\mathcal{T},\mathcal{F}}$ be the full subcategories of $\mathcal{T}$, $\mathcal{F}$ and $\mathcal{C}$, respectively, defined as follows:
	\begin{itemize}
		\item [(a)] $T\in\mathring{\mathcal{T}}$ if and only if $T\in\mathcal{T}$ and $\eta_T$ is an epimorphism;
		\item [(b)] $F\in\mathring{\mathcal{F}}$ if and only if $F\in\mathcal{F}$ and $\varepsilon_F$ is a monomorphism;
		\item [(c)] $B\in\mathcal{C}_{\mathcal{T},\mathcal{F}}$ if and only if $t(B)\in\mathring{\mathcal{T}}$, $f(B)\in\mathring{\mathcal{F}}$, the diagram given in 2.3(c) is a pullback, and the diagram given in 2.4(c) is a pushout. 
	\end{itemize}
Then $(\mathring{\mathcal{T}},\mathring{\mathcal{F}})$ is a torsion theory (that is, a pretorsion theory in the sense of \cite{[FF2020], [FFG2021]}) in $\mathcal{C}_{\mathcal{T},\mathcal{F}}$.
Moreover $\mathring{\mathcal{T}}=\mathcal{T}\cap\mathcal{C}_{\mathcal{T},\mathcal{F}}$,
$\mathring{\mathcal{F}}=\mathcal{F}\cap\mathcal{C}_{\mathcal{T},\mathcal{F}}$, and
$\mathcal{C}_{\mathcal{T},\mathcal{F}}$ is the largest full subcategory
$\mathcal{C}'$ of $\mathcal{C}$ such that every object $B$ of $\mathcal{C}'$
satisfies the conditions of Definition 2.5, with monomorphisms, epimorphisms, pullbacks, and pushouts involved being as in $\mathcal{C}$.
\end{teo}
\begin{proof}
	Thanks to Lemmas 2.3 and 2.4, all we need to show are the inclusions $\mathring{\mathcal{T}}\subseteq\mathcal{C}_{\mathcal{T},\mathcal{F}}$ and ${\mathring{\mathcal{F}}\subseteq\mathcal{C}_{\mathcal{T},\mathcal{F}}}$. The first of them is obvious since for $B\in\mathcal{T}$ each of our two diagrams becomes $$\xymatrix{B\ar@{=}[d]\ar[r]^-{\eta_B}&f(B)\ar@{=}[d]\\B\ar[r]_-{\eta_B}&f(B)}$$
	and the second inclusion follows by duality. The maximality assertion is immediate, since by (c) the objects of
$\mathcal{C}_{\mathcal{T},\mathcal{F}}$ are precisely those satisfying these
conditions.
\end{proof}

\section{The near-torsion theory formed by cokernel and kernel diagrams}

In this section we consider the near-torsion theory $(\mathcal{T},\mathcal{F})=(\mathcal{C},\mathcal{T},\mathcal{F},t,\varepsilon,f,\eta)$ associated with a pointed category $\mathcal{A}$ with kernels and cokernels as follows:
\begin{itemize}
	\item [-] $\mathcal{C}$ is the category of composable pairs $A=$ $$\xymatrix{A_0\ar[r]^-{d_0^A}&A_1\ar[r]^-{d_1^A}&A_2,}$$ where we will usually write $d_i^A=d_i$ $(i=0,1)$, and having $d_1d_0=0$.
	\item [-] $\mathcal{T}$ consists of all cokernel diagrams in $\mathcal{A}$, that is, $A\in\mathcal{T}\Leftrightarrow d^A_1=\mathrm{coker}(d^A_0)$.
	\item [-] $\mathcal{F}$ consists of all kernel diagrams in $\mathcal{A}$, that is, $A\in\mathcal{F} \Leftrightarrow d^A_0=\mathrm{ker}(d^A_1)$.
	\item [-] Accordingly,  the diagram $$\xymatrix{t(A)\ar[r]^-{\varepsilon_A}&A\ar[r]^-{\eta_A}&f(A)}$$ becomes the commutative diagram
	$$\xymatrix{A_0\ar[d]_{d_0}\ar@{=}[r]&A_0\ar[d]^{d_0}\ar[r]^-{d'_0}&\mathrm{Ker}(d_1)\ar[d]^{\mathrm{ker}(d_1)}\\A_1\ar[d]_{\mathrm{coker}(d_0)}\ar@{=}[r]&A_1\ar[d]^{d_1}\ar@{=}[r]&A_1\ar[d]^{d_1}\\\mathrm{Coker}(d_0)\ar[r]_-{d'_1}&A_2\ar@{=}[r]&A_2}$$ where $d'_0$ and $d'_1$ are determined by the suitable universal properties. 
\end{itemize}
Let us describe $\mathring{\mathcal{T}}$, $\mathring{\mathcal{F}}$, and $\mathcal{C}_{\mathcal{T},\mathcal{F}}$ in this case. For $F\in\mathcal{F}$ our display of $\varepsilon_F:t(F)\to F$ is $$\xymatrix{F_0\ar[d]_{d_0}\ar@{=}[r]&F_0\ar[d]^{d_0}\\F_1\ar[d]_{\mathrm{coker}(d_0)}\ar@{=}[r]&F_1\ar[d]^{d_1}\\\mathrm{Coker}(d_0)\ar[r]_-{d'_1}&F_2}$$ where $d_0$ is a kernel of $d_1$. Here $\varepsilon_F$ is a monomorphism if and only if $d_1=d_1^F$ admits a (normal epi, mono)-factorization. From this and the dual argument, we conclude:
\begin{itemize}
	\item [-] an object $T$ in $\mathcal{T}$ belongs to $\mathring{\mathcal{T}}$ if and only if $d_0^T$ admits an (epi, normal mono)-factorization;
	\item [-] an object $F$ in $\mathcal{F}$ belongs to $\mathring{\mathcal{F}}$ if and only if $d_1^F$ admits a (normal epi, mono)-factorization.
\end{itemize}

Next, for an object $B$ in $\mathcal{C}$, the diagram given in 2.3(c) becomes $$\xymatrix{B_0\ar[ddd]_{d_0}\ar@{=}[dr]\ar[rrr]&&&\mathrm{Ker}(d_1)\ar[ddd]^{\mathrm{ker}(d_1)}\ar@{=}[dr]\\&B_0\ar[ddd]_{d_0}\ar[rrr]^-{d'_0}&&&\mathrm{Ker}(d_1)\ar[ddd]^{\mathrm{ker}(d_1)}\\\\B_1\ar[ddd]_{\mathrm{coker}(d_0)}\ar@{=}[dr]\ar@{=}[rrr]&&&B_1\ar[ddd]^{\mathrm{coker}(\mathrm{ker}(d_1))}\ar@{=}[dr]\\&B_1\ar[ddd]_{d_1}\ar@{=}[rrr]&&&B_1\ar[ddd]^{d_1}\\\\\mathrm{Coker}(d_0)\ar[dr]\ar[rrr]&&&\mathrm{Coker}(\mathrm{ker}(d_1))\ar[dr]\\&B_2\ar@{=}[rrr]&&&B_2}$$ where the unlabeled arrows are determined by the suitable universal properties. This shows us that $B$ satisfies condition 2.3(c) if and only if the canonical morphism $\mathrm{Coker}(d_0)\to\mathrm{Coker}(\mathrm{ker}(d_1))$ is an isomorphism. Dually, $B$ satisfies 2.4(c) if and only if the canonical morphism $\mathrm{Ker}(\mathrm{coker}(d_0))\to\mathrm{Ker}(d_1)$ is an isomorphism. However, these two conditions are trivially equivalent to each other, and we obtain:
\begin{teo}
	An object $B$ in $\mathcal{C}$ belongs to $\mathcal{C}_{\mathcal{T},\mathcal{F}}$ if and only if it satisfies the following conditions:
	\begin{itemize}
		\item [(a)] the morphism $d_0^B$ admits an (epi, normal mono)-factorization;
		\item [(b)] the morphism $d_1^B$ admits a (normal epi, mono)-factorization;
		\item [(c)] the morphisms $d_0^B$ and $\mathrm{ker}(d_1^B)$ have the same cokernels, or, equivalently, the morphisms $d_1^B$ and $\mathrm{coker}(d_0^B)$ have the same kernels.\qed
	\end{itemize}
\end{teo}
\begin{cor}\label{Corollary-abelian}
	If the category $\mathcal{A}$ is abelian, then an object $B$ in $\mathcal{C}$ belongs to $\mathcal{C}_{\mathcal{T},\mathcal{F}}$ if and only if it is exact at $B_1$. In particular, if $\mathcal{A}$ is the category of modules over a ring, then this simply means that $d_0(B_0)=\mathrm{Ker}(d_1)$.\qed
\end{cor}  

\section{More on exactness in a non-abelian category}

We would like to rephrase Theorem 3.1 by saying that an object $B$ in $\mathcal{C}$ belongs to $\mathcal{C}_{\mathcal{T},\mathcal{F}}$ if and only if the sequence $B=$ $$\xymatrix{B_0\ar[r]^-{d_0}&B_1\ar[r]^-{d_1}&B_2}$$ is exact at $B_1$ in any case, not just in the abelian case. That is, we are introducing the following definition of exactness:
\begin{defi}
	\emph{A sequence $\ldots X\xrightarrow{u}Y\xrightarrow{v}Z\ldots$ of morphisms in a pointed category with kernels and cokernels with $vu=0$ is said to be} exact at $Y$ \emph{if the following conditions hold:}
	\begin{itemize}
		\item [(a)] \emph{$u$ admits an (epi, normal mono)-factorization;}
		\item [(b)] \emph{$v$ admits a (normal epi, mono)-factorization;}
		\item [(c)] \emph{the morphisms $u$ and $\mathrm{ker}(v)$ have the same cokernels, or, equivalently, the morphisms $v$ and $\mathrm{coker}(u)$ have the same kernels.}
	\end{itemize}		 
\end{defi}
The purpose of this section is only to add some clarifying observations to Definition 4.1.
\begin{prop}
	Let $E=(\ldots X\xrightarrow{u}Y\xrightarrow{v}Z\ldots)$ be a sequence of morphisms in a pointed category with kernels and cokernels, with $vu=0$, and $f:X\to\mathrm{Ker}(v)$ be the morphism determined by the universal property of the kernel of $v$. Then:
	\begin{itemize}
		\item [(a)] if $f$ is an epimorphism then conditions $(a)$ and $(c)$ of Definition $4.1$ hold;
		\item [(b)] if $u$ admits an (epi, normal mono)-factorization and $E$ is exact at $Y$, then $f$ is an epimorphism.
	\end{itemize}
\end{prop}
\begin{proof}
	(a): Since $u=\mathrm{ker}(v)f$ and $f$ is an epimorphism, $u$ and $\mathrm{ker}(v)$ have the same cokernels.
	
	(b): Consider the commutative diagram $$\xymatrix{X\ar@/^-0.5cm/[ddr]_e\ar[dr]|f\ar[rr]^-u&&Y\ar[dr]|{\mathrm{coker}(u)}\ar[rr]^-v&&Z\\&\mathrm{Ker}(v)\ar[ur]|{\mathrm{ker}(v)}&&\mathrm{Coker}(u)\\&S\ar@/^-0.5cm/[uur]_m\ar[u]|g}$$ in which $u=me$ is the (epi, normal mono)-factorization of $u$ and $g$ (as well as $f$) is determined by the universal property of the kernel of $v$. Since $e$ is an epimorphism, $u$ and $m$ have the same cokernels. Since $E$ is exact at $Y$, it follows that $\mathrm{ker}(v)$ and $m$ have the same cokernels. Since $\mathrm{ker}(v)$ and $m$ are normal monomorphisms with the same cokernels, $g$ is an isomorphism. Since $f=ge$, it follows that $f$ is an epimorphism.
\end{proof}
Dually, we obtain:
\begin{prop}
	Let $E=(\ldots X\xrightarrow{u}Y\xrightarrow{v}Z\ldots)$ be a sequence of morphisms in a pointed category with kernels and cokernels, with $vu=0$, and $h:\mathrm{Coker}(u)\to Z$ be the morphism determined by the universal property of the cokernel of $u$. Then:
	\begin{itemize}
		\item [(a)] if $h$ is a monomorphism then conditions $(b)$ and $(c)$ of Definition $4.1$ hold;
		\item [(b)] if $v$ admits a (normal epi, mono)-factorization and $E$ is exact at $Y$, then $h$ is a monomorphism.
	\end{itemize}
\end{prop}
\begin{cor}
	A sequence $\ldots X\xrightarrow{u}Y\xrightarrow{v}Z\ldots$ of morphisms in a pointed category with kernels and cokernels with $vu=0$ is exact at $Y$ if and only if the following conditions hold:
	\begin{itemize}
		\item [(a)] $u$ admits an (epi, normal mono)-factorization;
		\item [(b)] $v$ admits a (normal epi, mono)-factorization;
		\item [(c)] the canonical morphism $X\to\mathrm{Ker}(v)$ is an epimorphism, or, equivalently, the canonical morphism $\mathrm{Coker}(u)\to Z$ is a monomorphism. 
	\end{itemize}	
\end{cor}
\begin{example}
	\emph{Let $\mathcal{A}$ be a pointed category with kernels and cokernels. For a sequence $E=(\ldots X\xrightarrow{u}Y\xrightarrow{v}Z\ldots)$ in $\mathcal{A}$ with $vu=0$, we have:}
	\begin{itemize}
		\item [(a)] \emph{If the ambient category $\mathcal{A}$ is abelian, or, more generally, exact in the sense of \cite{[P1962]} (see also \cite{[G2012]}), then all monomorphisms and all epimorphisms in it are normal, and each morphism in it admits an (epi, mono)-factorization. In this case conditions 4.4(a) and 4.4(b) hold automatically, and we can omit them (cf. Corollary 3.2). More generally, this can also be done, of course, when we only know that every morphism in $\mathcal{A}$ admits both (epi, normal mono)- and (normal epi, mono)-factorizations. For instance, this is the case if both $\mathcal{A}$ and $\mathcal{A}^{\mathrm{op}}$ are} normal \emph{in the sense of \cite{[J2010]}, or, in particular,} homological \emph{in the sense of \cite{[BB2004]}. A further special case is the case of $\mathcal{A}$ being} semiabelian \emph{in the sense of \cite{[R1969]}, which includes, e.g. the categories of torsion free abelian groups, of Banach spaces, and several other well-known examples.}
		\item [(b)] \emph{When $\mathcal{A}$ is a pointed regular category with kernels and cokernels, if $E$ is exact in the sense of Definition 4.1.7 in \cite{[BB2004]}, then it satisfies all conditions of Corollary 4.4. The converse implication does not hold in general, even when $\mathcal{A}$ is semi-abelian in the sense of \cite{[JMT2002]}, unless every epimorphism in $\mathcal{A}$ is regular (as e.g. in the category of groups). However, an infinite sequence $$\ldots\to X_2\to X_1\to X_0$$ in any pointed regular category is exact in each $X_i$ $(i=1,2,\ldots)$ in the sense of our Definition 4.1 if and only if it is exact in the sense of Definition 4.1.7 in \cite{[BB2004]}}.
        \item [(c)] \emph{When $\mathcal{A}$ is the category of pointed sets (which is pointed regular), `classically' $E$ is said to be exact at $Y$ if $u(X)$ is equal to the inverse image of the distinguished point of $Z$ (see e.g. \cite{[G1955], [GZ1967], [S1997]}). This is equivalent to condition 4.4(c), and here 4.4(a) (but not 4.4(b)) holds automatically. On the other hand, here every epimorphism is regular, and our Definition 4.1 is equivalent to Definition 4.1.7 in \cite{[BB2004]}}.  
	\end{itemize}
\end{example}

\section{Expected further developments}

\textbf{5.1.} We expect various counterparts of constructions of Section 3 and Theorem 3.1 for other notions of exactness. It could be exactness with respect to an ideal of morphisms, introduced by C. Ehresmann in \cite{[Ehr1964]} and mainly developed in the work of M. Grandis, which is also used in \cite{[GJ2020]}. Or, one could replace composable pairs considered in Section 3 with \textit{forks} $$\xymatrix{A_0\ar@<0.5ex>[r]^-{d_0^A}\ar@<-0.5ex>[r]_-{d_0'^A}&A_1\ar[r]^-{d_1^A}&A_2,}$$ with $d_1^A d_0^A=d_1^A d_0'^A$ instead of $d_1^A d_0^A=0$. Or, one could use the more general context of regular multi-pointed categories introduced in \cite{[GJU2012]}.

\textbf{5.2.} In many known examples of non-pointed torsion theories one could search for the situation where $\mathcal{C}_{\mathcal{T},\mathcal{F}}$ turns out to be the ambient category that is inside a category $\mathcal{C}$, which appears more naturally (but gives merely a near-torsion theory).

\textbf{5.3.} The diagrams in 2.4(c) and in Lemma 4.10 of \cite{[BG2006]} are identical, although they appear in different contexts; in fact in the context of \cite{[BG2006]}, the diagrams in 2.3(c) and in 2.4(c) also coincide. Having in mind also Proposition 2.3 in \cite{[GJM2013]}, this requires a further study.

\textbf{5.4.} The precise relationship with category-theoretic notions of radicals introduced in \cite{[GJM2013]} and \cite{[GJM2017]} will hopefully be established.
 
	{}
	
\end{document}